\documentclass[12pt,leqno]{article}
\usepackage{amsfonts}
\usepackage{amsmath, amsthm, amsfonts, amssymb, color}
\usepackage{mathrsfs}
\usepackage{color}
\newtheorem{thm}{Theorem}[section]

\newtheorem{lem}[thm]{Lemma}

\newtheorem{rem}[thm]{Remark}
\theoremstyle{definition}

\newcommand{\scr}[1]{\mathscr #1}
\definecolor{wco}{rgb}{0.5,0.2,0.3}

\numberwithin{equation}{section} \theoremstyle{remark}

\newcommand{\ua}{\uparrow}

\title{{\bf  Exponential Convergence for Functional SDEs with H\"{o}lder Continuous Drift}\footnote{Supported in
 part by  NNSFC (11801406).}}
\author{
{\bf   Xing Huang }\\
\footnotesize{ Center for Applied Mathematics, Tianjin
University, Tianjin 300072, China}\\
\footnotesize{  xinghuang@tju.edu.cn}}
\begin{document}
\allowdisplaybreaks
\def\R{\mathbb R}  \def\ff{\frac} \def\ss{\sqrt} \def\B{\mathbf
B} \def\W{\mathbb W}
\def\N{\mathbb N} \def\kk{\kappa} \def\m{{\bf m}}
\def\ee{\varepsilon}\def\ddd{D^*}
\def\dd{\delta} \def\DD{\Delta} \def\vv{\varepsilon} \def\rr{\rho}
\def\<{\langle} \def\>{\rangle} \def\GG{\Gamma} \def\gg{\gamma}
  \def\nn{\nabla} \def\pp{\partial} \def\E{\mathbb E}
\def\d{\text{\rm{d}}} \def\bb{\beta} \def\aa{\alpha} \def\D{\scr D}
  \def\si{\sigma} \def\ess{\text{\rm{ess}}}
\def\beg{\begin} \def\beq{\begin{equation}}  \def\F{\scr F}
\def\Ric{\text{\rm{Ric}}} \def\Hess{\text{\rm{Hess}}}
\def\e{\text{\rm{e}}} \def\ua{\underline a} \def\OO{\Omega}  \def\oo{\omega}
 \def\tt{\tilde} \def\Ric{\text{\rm{Ric}}}
\def\cut{\text{\rm{cut}}} \def\P{\mathbb P} \def\ifn{I_n(f^{\bigotimes n})}
\def\C{\scr C}      \def\aaa{\mathbf{r}}     \def\r{r}
\def\gap{\text{\rm{gap}}} \def\prr{\pi_{{\bf m},\varrho}}  \def\r{\mathbf r}
\def\Z{\mathbb Z} \def\vrr{\varrho} \def\ll{\lambda}
\def\L{\scr L}\def\Tt{\tt} \def\TT{\tt}\def\II{\mathbb I}
\def\i{{\rm in}}\def\Sect{{\rm Sect}}  \def\H{\mathbb H}
\def\M{\scr M}\def\Q{\mathbb Q} \def\texto{\text{o}} \def\LL{\Lambda}
\def\Rank{{\rm Rank}} \def\B{\scr B} \def\i{{\rm i}} \def\HR{\hat{\R}^d}
\def\to{\rightarrow}\def\l{\ell}\def\iint{\int}
\def\EE{\scr E}\def\Cut{{\rm Cut}}
\def\A{\scr A} \def\Lip{{\rm Lip}}
\def\BB{\scr B}\def\Ent{{\rm Ent}}\def\L{\scr L}
\def\R{\mathbb R}  \def\ff{\frac} \def\ss{\sqrt} \def\B{\mathbf
B}
\def\N{\mathbb N} \def\kk{\kappa} \def\m{{\bf m}}
\def\dd{\delta} \def\DD{\Delta} \def\vv{\varepsilon} \def\rr{\rho}
\def\<{\langle} \def\>{\rangle} \def\GG{\Gamma} \def\gg{\gamma}
  \def\nn{\nabla} \def\pp{\partial} \def\E{\mathbb E}
\def\d{\text{\rm{d}}} \def\bb{\beta} \def\aa{\alpha} \def\D{\scr D}
  \def\si{\sigma} \def\ess{\text{\rm{ess}}}
\def\beg{\begin} \def\beq{\begin{equation}}  \def\F{\scr F}
\def\Ric{\text{\rm{Ric}}} \def\Hess{\text{\rm{Hess}}}
\def\e{\text{\rm{e}}} \def\ua{\underline a} \def\OO{\Omega}  \def\oo{\omega}
 \def\tt{\tilde} \def\Ric{\text{\rm{Ric}}}
\def\cut{\text{\rm{cut}}} \def\P{\mathbb P} \def\ifn{I_n(f^{\bigotimes n})}
\def\C{\scr C}      \def\aaa{\mathbf{r}}     \def\r{r}
\def\gap{\text{\rm{gap}}} \def\prr{\pi_{{\bf m},\varrho}}  \def\r{\mathbf r}
\def\Z{\mathbb Z} \def\vrr{\varrho} \def\ll{\lambda}
\def\L{\scr L}\def\Tt{\tt} \def\TT{\tt}\def\II{\mathbb I}
\def\i{{\rm in}}\def\Sect{{\rm Sect}}  \def\H{\mathbb H}
\def\M{\scr M}\def\Q{\mathbb Q} \def\texto{\text{o}} \def\LL{\Lambda}
\def\Rank{{\rm Rank}} \def\B{\scr B} \def\i{{\rm i}} \def\HR{\hat{\R}^d}
\def\to{\rightarrow}\def\l{\ell}
\def\8{\infty}\def\I{1}\def\U{\scr U}
 \def\T{\scr T}
\maketitle

\begin{abstract} Applying Zvonkin's transform, the exponential convergence in Wasserstein distance for a class of functional SDEs with H\"{o}lder continuous drift is obtained. This combining with log-Harnack inequality implies the same convergence in the sense of entropy, which also yields the convergence in total variation norm by Pinsker's inequality.
\end{abstract} \noindent
 \vskip 2cm

\section{Introduction}
Consider the SDE on $\mathbb{R}^d$
\begin{align}\label{EE0}\d X(t)=b(X(t))\d t+\d W(t),
\end{align}
where $b:\mathbb{R}^d\to \mathbb{R}^d$, $W$ is a $d$-dimensional Brownian motion on some complete filtration probability space. If the dissipative condition
$$\langle b(x)-b(y),x-y\rangle\leq -\kappa_0|x-y|^2, \ \ x,y \in\mathbb{R}^d $$
holds for some $\kappa_0>0$, then SDE \eqref{EE0} has a unique solution and the associated semigroup has exponential convergence in Wasserstein distance, see \cite{BGG}. \cite{LW} has proved exponential convergence in $L^p$ Wasserstein distance under the condition
$$\langle b(x)-b(y),x-y\rangle\leq
      K_1|x-y|^21_{\{|x-y|\leq L\}}-K_2|x-y|^21_{\{|x-y|>L\}}
$$
for some positive constants $K_1, K_2, L$.
\cite{W00} extends this to diffusion semi-
groups on Riemannian manifold with negative curvature. Concerning the functional SDE, the method of reflection coupling used in \cite{LW} and \cite{W00} is so difficult to construct that we can not obtain similar results as in \cite{LW} and \cite{W00}. Recently, the exponential convergence in the sense of Wasserstein distance and total variation norm has been obtained in \cite{BWY150,BWY15} for a class of functional SDEs/SPDEs with regular coefficients and additive noise, where exponential convergence in total variation norm is proved due to the gradient-$L^2$ estimate
$$|\nabla P_t f|^2\leq C P_t |f|^2, \ \ t>r_0, f\in\B_b(\C),$$
see \cite{BWY150,BWY15} for more details.
On the other hand, using Zvonkin's transform \cite{AZ}, the strong well-posedness of SDEs is proved for SDEs with singular drifts, see \cite{GM,KR,W,Z,Z2,Z3}. For the functional SDEs with singular drift, \cite{H} proved the existence and uniqueness. In infinite dimension, \cite{HW,HZ} obtain the existence and uniqueness of the mild solution for a class of semi-linear functional SPDEs with Dini continuous drift and establish the Harnack inequality. In this paper, we will study the exponential convergence for functional SDEs with H\"{o}lder continuous drift.

Recall that for two probability measures $\mu,\nu$ on   some measurable space $(E,\scr F)$, the entropy and total variation norm are defined as follows:
$$\Ent(\nu|\mu):= \beg{cases} \int (\log \ff{\d\nu}{\d\mu})\d\nu, \ &\text{if}\ \nu\ \text{ is\ absolutely\ continuous\ with\ respect\ to}\ \mu,\\
 \infty,\ &\text{otherwise;}\end{cases}$$ and
$$\|\mu-\nu\|_{var} := \sup_{A\in\F}|\mu(A)-\nu(A)|.$$ By Pinsker's inequality (see \cite{CK, Pin}),
\beq\label{ETX} \|\mu-\nu\|_{var}^2\le \ff 1 2 \Ent(\nu|\mu),\ \ \mu,\nu\in \scr P(E),\end{equation}
here $\scr P(E)$ denotes all probability measures on $(E,\F)$.  Indeed, these two estimates correspond to  the log-Harnack inequality  for the associated semigroups, see Lemma \ref{Llog} below for details.

When $E$ is a Polish space, in particular, $E=\C$ in our frame, which will be defined in the sequel,
let
$$\scr P_2 := \big\{\mu\in \scr P(\C): \mu(\|\cdot\|_{\infty}^2)<\infty\big\}.$$    It is well known that
$\scr P_2$ is a Polish space under the Wasserstein distance
$$\W_2(\mu,\nu):= \inf_{\pi\in \mathbf{C}(\mu,\nu)} \bigg(\int_{\C\times\C} \|\xi-\eta\|_{\infty}^2 \pi(\d \xi,\d \eta)\bigg)^{\ff 1 {2}},\ \ \mu,\nu\in \scr P_{2},$$ where $\mathbf{C}(\mu,\nu)$ is the set of all couplings of $\mu$ and $\nu$.  Moreover,   the topology induced by $\W_2$ on $\scr P_2$ coincides with the weak topology.

The purpose of this paper is to establish the exponential convergence in the sense of  Wasserstein distance, the entropy and total variation norm respectively for functional SDEs with H\"{o}lder continuous drift, which is much weaker than the Lipschitz  condition.

Throughout the paper, we fix $r_0>0$ and consider the path space $\C:= C([-r_0,0];\R^d)$ equipped with the uniform norm $\|\xi\|_\infty:=\sup_{s\in [-r_0.0]}|\xi(s)|. $ For any $f\in C([-r_0,\infty);\mathbb{R}^d)$, $t\geq 0$, let $f_{t}(s)=f(t+s), s\in [-r_0,0]$. Then $f_{t} \in \C$. $\{f_t\}_{t\geq 0}$ is called the segment process of $f$.
Consider the following functional SDE on $\R^d$:
\beq\label{E1} \d X(t)=-\beta X(t)\d t+\{b(X(t))+B(X_t)\}\d t+\d W(t),\end{equation}
where $W(t)$ is a $d$-dimensional Brownian motion on a complete filtration probability space $(\OO,\{\F_t\}_{t\ge 0},\F, \P)$, $\beta>0$ and
$$b:\R^d\to \R^d,\ \ B:\C\to \R^d$$ are measurable.

The remainder of the paper is organized as follows. In Section 2 we summarize the main results of the paper; In section 3, we give precise estimate for Zvonkin's transform and the main results are proved in Section 4.

\section{Main Results}
Throughout this paper, we make the following assumptions:
\beg{enumerate} \item[{\bf (H1)}] $b$ is bounded, i.e.
\begin{equation}\label{bao3} \|b\|_{\infty}<\infty.
\end{equation}
Moreover, there exist constants $\kappa>0$ and $\alpha\in(0,1)$ such that
 \begin{align}\label{b-in}
|b(x)-b(y)|\leq \kappa|x-y|^\alpha, \ \ x,y \in\mathbb{R}^d.
\end{align}
\item[{\bf (H2)}] There exists a constant $\lambda_B>0$ such that
\begin{align}\label{B-Lip}\|B\|_{\infty}<\infty, \ \ |B(\xi)-B(\eta)|\leq\lambda_B\|\xi-\eta\|_\infty, \ \ \ \ \xi,\eta\in\C.
\end{align}
\end{enumerate}

Since H\"{o}lder continuity is stronger than Dini continuity, according to \cite[Theorem 2.1]{HZ} for $\mathbb{H}=\mathbb{R}^d$, under {\bf(H1)}-{\bf(H2)}, the SDE \eqref{E1} has a unique non-explosive solution denoted by $X_t^\xi$ with $X_0=\xi$ and
\begin{align}\label{sup0}\mathbb{E}\sup_{t\in[0,T]}\|X_t^\xi\|_{\infty}<\infty, \ \ T>0.
\end{align} Let $P_t(\xi,\d \eta)$ be the distribution of $X_t^\xi$, and $$P_tf(\xi):=\int_{\C}f(\eta)P_t(\xi,\d \eta), \ \ f\in\B_b(\C).$$
Moreover, for any $\nu\in\scr P_2$, let $\nu P_t:=\int_{\C}P_t(\xi,\cdot)\nu(\d \xi)$. Then $\nu P_t$ is the distribution of the solution $X_t$ to \eqref{E1} from initial distribution $\nu$. In particular, $P_t(\xi,\cdot)=\delta_\xi P_t$, here, $\delta_\xi$ is the Dirac measure on $\xi$. 

The lemma below gives the estimate of $\Ent(P_{t}(\xi,\cdot)|P_{t}(\eta,\cdot))$ and $\|P_t(\xi,\cdot)-P_t(\eta,\cdot)\|_{var}$ respectively.
\begin{lem}\label{Llog} Assume {\bf(H1)}-{\bf(H2)}. Then the log-Harnack inequality holds, i.e.
\beq\label{log0}
P_t\log f(\eta)\leq \log P_t f(\xi)+\frac{C(t)}{(t-r_0)\wedge 1}\|\xi-\eta\|_{\infty}^2,\ \ t>r_0, \xi,\eta\in\C, f\in\B^+_b(\mathbb{R}^d)
\end{equation}
for some function $C:(r_0,\infty)\to(0,\infty)$. Thus, for any $t>r_0$, $P_t(\xi,\cdot)$ is equivalent to $P_t(\eta,\cdot)$. Moreover,
\begin{align}\label{Ent}\Ent(P_{t}(\xi,\cdot)|P_{t}(\eta,\cdot))=P_t\left\{\log\frac{\d P_t(\xi,\cdot)}{\d P_t(\eta,\cdot)}\right\}(\xi)\leq  \frac{C(t)}{(t-r_0)\wedge 1}\|\xi-\eta\|_{\infty}^2, \end{align}
and
\beq\label{TT}\|P_t(\xi,\cdot)-P_t(\eta,\cdot)\|_{var}^2\le  \frac{C(t)}{2(t-r_0)\wedge 2}\|\xi-\eta\|_{\infty}^2. \end{equation}
\end{lem}
\begin{proof} The log-Harnack inequality \eqref{log0} is a known result in \cite[Theorem 2.2]{HZ} with $\mathbb{H}=\mathbb{R}^d$, see also \cite[Theorem 2.4]{HW1} for log-Harnack inequality of the path-distribution dependent SDEs with Dini continuous drift. Combining the definition of  $\Ent(P_{t}(\xi,\cdot)|P_{t}(\eta,\cdot))$ and \cite[Theorem 1.4.2]{Wbook}, we obtain \eqref{Ent}. Finally, \eqref{TT} follows from \eqref{ETX} and \eqref{Ent}.
\end{proof}
\begin{rem} \label{W2} For any $\nu,\tilde{\nu}\in\scr P_2$ and $\pi\in\mathbf{C}(\nu,\tilde{\nu})$, taking  expectation on both sides of \eqref{log0} with respect to $\pi$, we have for any $t>r_0$,
\begin{align*}
\int_{\C^2}P_t\log f(\eta)\pi(\d \xi,\d \eta)\leq \int_{\C^2}\log P_t f(\xi)\pi(\d \xi,\d \eta)+\frac{C(t)}{(t-r_0)\wedge 1}\int_{\C^2}\|\xi-\eta\|_{\infty}^2\pi(\d \xi,\d \eta).
\end{align*}
Jensen's inequality and the definition of $\mathbb{W}_2$ imply that
\beq\label{log1}
(\tilde{\nu}P_t)(\log f)\leq \log (\nu P_t)(f)+\frac{C(t)}{(t-r_0)\wedge 1}\mathbb{W}_2(\nu,\tilde{\nu})^2,\ \ t>r_0.
\end{equation}
Then we have
\begin{align}\label{Ent00}\Ent(\nu P_t|\tilde{\nu}P_{t})\leq  \frac{C(t)}{(t-r_0)\wedge 1}\mathbb{W}_2(\nu,\tilde{\nu})^2, \ \  t>r_0,\end{align}
and
\beq\label{TT00}\|\nu P_t-\tilde{\nu}P_{t}\|_{var}^2\le  \frac{C(t)}{2(t-r_0)\wedge 2}\mathbb{W}_2(\nu,\tilde{\nu})^2, \ \ t>r_0. \end{equation}
\end{rem}
Let $\Gamma(\cdot)$ be the Gamma function. In order to state our main result, we firstly give some notations.

For any $\lambda>0, \delta,\varepsilon\in(0,1)$, let
\begin{equation}\label{z30}
\lambda_0(\delta):=\left(\frac{1+\delta}{\delta}\sqrt{\pi}\|b\|_{\infty}\right)^2 \vee\left(\frac{1+\delta}{\delta}\|b\|_{\infty}\right);
\end{equation}
\begin{equation}\label{z3}\begin{split}
\Upsilon(\lambda,\delta):&=2\Gamma(\frac{\alpha}{2}) \lambda^{-\frac{\alpha}{2}}\left((1+\delta)\kappa+ 2\delta\|b\|_{\infty}+4(1+\delta)(3+\lambda^{-1})\|b\|^2_{\8}\right);
\end{split}\end{equation}
and 
\begin{align*}\Lambda(\lambda,\varepsilon,\delta):&=\frac{2}{1-\varepsilon}\left((1+\delta)\delta\lambda +\beta\delta+(1+\delta)^2\lambda_B+(1+\delta)\Upsilon(\lambda,\delta) \|B\|_{\infty}\right)\\
&\qquad+\frac{1}{1-\varepsilon}\left(d+\frac{4}{(1-\delta)^2\varepsilon}\right)\Upsilon(\lambda,\delta)^2,
\end{align*}
The main result in this paper is the following theorem.
\begin{thm}\label{TE} Assume {\bf(H1)}-{\bf(H2)}, 
then the following assertions hold.
\begin{enumerate}
\item[(1)] The following estimate holds:
\begin{align*}\mathbb{E}\|X_t^\xi-X_t^\eta\|^2_{\infty}&\leq \inf_{\delta,\varepsilon\in(0,1), \lambda> \lambda_0(\delta)}\frac{\e^{\frac{2\beta}{(1-\delta)^2} r_0}}{1-\varepsilon}\exp\left\{\frac{\e^{\frac{2\beta}{(1-\delta)^2} r_0}}{(1-\delta)^2}\left(\Lambda(\lambda,\varepsilon,\delta)-2\beta \e^{-\frac{2\beta}{(1-\delta)^2} r_0}\right)t\right\}\\
&\qquad\qquad\times\|\xi-\eta\|_{\infty}^2, \ \ t>0, \xi,\eta\in\C.
\end{align*}
\item[(2)] If there exists $\tilde{\delta},\tilde{\varepsilon}\in(0,1), \tilde{\lambda}> \lambda_0(\tilde{\delta})$ such that
\begin{align}\label{nl}
\Lambda(\tilde{\lambda},\tilde{\varepsilon},\tilde{\delta})<2\beta \e^{-\frac{2\beta}{(1-\tilde{\delta})^2} r_0},
\end{align}
then \eqref{E1} has a unique invariant probability measure $\mu$  and for any $t_0>r_0$,
\begin{align*}\max\left\{\mathbb{W}_2(P_t(\xi,\cdot),\mu), \Ent(P_{t}(\xi,\cdot)|\mu), \|P_t(\xi,\cdot)-\mu\|_{var}\right\}\leq \kappa_1(\xi)\e^{-\kappa_2 t},\ \ \xi\in\C,t>t_0
\end{align*}
for some constants $\kappa_1(\xi),\kappa_2>0$, here $\kappa_1(\xi)$ means it depends on $\xi$.
\end{enumerate}
\end{thm}
\begin{rem} \label{st} When $b=0$, we have $\lambda_0(\delta)=\Upsilon(\lambda,\delta)=0$, and $$\Lambda(\lambda,\varepsilon,\delta)=\frac{2}{1-\varepsilon}\left((1+\delta)\delta\lambda +\beta\delta+(1+\delta)^2\lambda_B\right).$$ 
In this case, \eqref{nl} holds provided $\lambda_B<\beta \e^{-2\beta r_0}$, which was used in \cite[Lemma 2.4]{BWY150}. In other words, Theorem \ref{TE}(2) covers the result in \cite[Lemma 2.4]{BWY150} with $Z(x)=-\beta x, \sigma=I$ there. On the other hand, there are a lot of examples where $\W_2(\mu_n,\mu_0)$  goes to $0$ as $n$ goes to infinity, but $\mu_n$ is singular with respect to $\mu_0$ such that $\Ent(\mu_n|\mu_0)=\infty$ and $\|\mu_n-\mu_0\|_{var}=1.$ Thus, the assertion in Theorem \ref{TE}(2) is not trivial.
\end{rem}
\begin{rem}\label{ex} From the proof of Theorem \ref{TE} below, we can extend \eqref{E1} to the following SDE
\beq\label{E1''0} \d X(t)=A X(t)\d t+\{b(X(t))+B(X_t)\}\d t+\sigma\d W(t),\end{equation}
\end{rem}
where $A,\sigma\in\mathbb{R}^d\otimes\mathbb{R}^d$, $A$ is negative definite and self-joint, and $\sigma$ is inverse. Since in this case, $\Lambda(\lambda,\varepsilon,\delta)$, $\Upsilon(\lambda,\delta)$ are more sophisticated, we do not give them here.  
\section{Precise Estimate for Zvonkin's Transform}
Since $b$ is singular, we need to construct a regular transform to remove $b$. To this end, for any $\lambda>0$, consider the following equation
\begin{equation}\label{b2}
u=\int_0^\infty\e^{-\lambda t}P_{t}^0\{b+\nabla_{b}u\}\d t,
\end{equation}
where the semigroup $(P_{t}^0)_{t>0}$ is generated by $(Z_t^{x})_{t\geq 0}$ which solves the SDE
\begin{equation}\label{b0}
\d Z_t^{x}=-\beta Z_t^x+\d W(t), \ \ Z_0^{x}=x.
\end{equation}
The following lemma gives a precise estimate for the solution to \eqref{b2}, and it is very important in the proof of the exponential convergence.
\begin{lem}\label{L2.1}
 {\rm Under {\bf (H1)}, for any $\delta\in(0,1)$, there exists a constant $\lambda_0(\delta)>0$ defined in \eqref{z30} such that for any $ \lambda>\lambda_0(\delta)$, the following assertions hold. 
\begin{enumerate}
\item[\bf{(i)}]  The equation \eqref{b2} has a unique
strong solution $\mathbf{u}^\lambda\in C_b^1(\R^d;\R^d)$;
\item[\bf{(ii)}]
$\|\nn \mathbf{u}^\lambda\|_{\infty}\le \delta$;
\item[\bf{(iii)}]
$ \|\nn^2
\mathbf{u}^\lambda\|_{\8}\le\Upsilon(\lambda,\delta)$ with $\Upsilon(\lambda,\delta)$ defined in \eqref{z3}.
\end{enumerate}
 }
\end{lem}

\begin{proof} Firstly, it is easy to see that $$\nn_\eta
Z_t^{x}=\e^{- \beta t}\eta, \ \ \nn_{\eta'}\nn_\eta
Z_t^{x}=0,\ \ t\geq 0. $$
This combining the Bismut formula \cite[(2.8)]{W} implies that
\begin{equation}\label{h0}\begin{split}
\nn_\eta
P_{t}^0f(x)&=\E\Big(\frac{f(Z_t^{x})}{t}\int_0^t\<\nn_\eta
Z_r^{x},\d W(r)\>\Big)\\
&=\E\Big(\frac{f(Z_t^{x})}{t}\int_0^t\<\e^{-\beta r}\eta,\d W(r)\>\Big),~~~f\in\B_b(\R^d), t>0.
\end{split}\end{equation}
By the Cauchy-Schwartz inequality and It\^o's isometry, we obtain that
\begin{equation}\label{b7}
\begin{split}
|\nn_\eta
P_{t}^0f|^2(x)&\le
\ff{|\eta|^2P_{t}^{0}f^{2}(x)}{t},~~~f\in\B_b(\R^d), t>0.
\end{split}
\end{equation}
{\bf(i)} Let $\scr H=C_{b}^{1}(\mathbb{R}^d; \mathbb{R}^d))$, which is a Banach space under the norm
\begin{equation*}\begin{split}
\|u\|_{\scr H}:&=\|u\|_{\infty}+\|\nabla u\|_{\infty}.
\end{split}\end{equation*}
For any $\lambda>0$, $u\in\scr H$, define
\begin{equation*}
(\Gamma^\lambda u)(x)=\int_{0}^{\infty} \e^{-\lambda t}P_{t}^{0}(\nabla_{b}u+b)(x)\d t.
\end{equation*}
Then we claim $\Gamma^\lambda\scr H\subset\scr H$ for any $\lambda>0$. In fact, for any $u\in\scr H$, it holds that
\begin{equation*}\beg{split}
\|\Gamma^\lambda u\|_{\infty}&=\sup_{x\in\mathbb{R}^d}\left|\int_{0}^{\infty} \e^{-\lambda t}P_{t}^{0}(\nabla_{b}u+b)(x)\d t\right|\\
&\leq (\|b\|_{\infty}\|\nabla u\|_{\infty}+\|b\|_{\infty})\int_{0}^{\infty} \e^{-\lambda t}\d t\\
&\leq \frac{\|b\|_{\infty}\|\nabla u\|_{\infty}+\|b\|_{\infty}}{\lambda}<\infty.
\end{split}\end{equation*}
By \eqref{b7}, we have
\begin{equation}\beg{split}\label{gu}
\|\nabla\Gamma^\lambda u\|_{\infty}&=\sup_{x\in\mathbb{R}^d, |\eta|\leq 1}\left|\int_{0}^{\infty} \e^{-\lambda t}\nabla_{\eta} P_{t}^{0}(\nabla_{b}u+b)(x)\d t\right|\\
&\leq \int_{0}^{\infty} \frac{\e^{-\lambda t}}{\sqrt{t}}(\|b\|_{\infty}\|\nabla u\|_{\infty}+\|b\|_{\infty})\d t\\
&\leq \|(\|b\|_{\infty}\|\nabla u\|_{\infty}+\|b\|_{\infty})\int_{0}^{\infty} \frac{\e^{-\lambda t}}{\sqrt{t}}\d t\\
&\leq \frac{\sqrt{\pi}(\|b\|_{\infty}\|\nabla u\|_{\infty}+\|b\|_{\infty})}{\sqrt{\lambda}}<\infty.
\end{split}\end{equation}
So, $\Gamma^\lambda\scr H\subset\scr H$ for any $\lambda>0$. Next, by the fixed-point theorem, it suffices to show that for large enough $\lambda>0$, $\Gamma^\lambda$ is contractive on $\scr H$. To do this, for any $u$, $\tilde{u}\in\scr H$, similarly to the estimates of $\|\Gamma^\lambda u\|_\infty$ and $\|\nabla \Gamma^\lambda u\|_\infty$ above, we obtain that
\begin{equation}\beg{split}\label{u-nu}
&\|\Gamma^\lambda u-\Gamma^\lambda \tilde{u}\|_{\infty}\leq \frac{\|b\|_{\infty}}{\lambda}\|\nabla u-\nabla \tilde{u}\|_{\infty},\\
&\|\nabla(\Gamma^\lambda u-\Gamma^\lambda \tilde{u})\|_{\infty}\leq \frac{\sqrt{\pi}\|b\|_{\infty}}{\sqrt{\lambda}}\|\nabla u-\nabla \tilde{u}\|_{\infty}.
\end{split}\end{equation}
Taking $\lambda>0$ satisfying
\begin{align}\label{lam}
\lambda>\lambda_0(\delta)=\left(\frac{1+\delta}{\delta}\sqrt{\pi}\|b\|_{\infty}\right)^2 \vee\left(\frac{1+\delta}{\delta}\|b\|_{\infty}\right),
 \end{align}
 then $\Gamma^\lambda$ is contractive on $\scr H$, which implies that \eqref{b2} has a unique solution $\mathbf{u}^\lambda\in C_{b}^{1}(\mathbb{R}^d; \mathbb{R}^d))$ by the fixed-point theorem.
Moreover,  by \eqref{gu} and \eqref{lam}, {\bf(i)} holds for $\lambda>\lambda_0(\delta)$.

{\bf(ii)} For any $\lambda>\lambda_0(\delta)$, one infers from \eqref{b2} and \eqref{b7} that
\begin{equation*}
\begin{split}
\|\nn \mathbf{u}^\ll\|_\infty&\le\int_0^\infty\e^{-\ll t}\|\nn
P_{t}^0\{b+\nabla_{b}\mathbf{u}^\ll\}\|_\infty\d t\\
&\le\|(1+\|\nn \mathbf{u}^\lambda\|_{\8})\|b\|_{\8}
\int_0^\infty\frac{\e^{-\ll t}}{\ss{t}}\d t\\
&\le\,\lambda^{-\frac{1}{2}}\ss\pi\|b\|_{\8}(1+\|\nn
\mathbf{u}^\lambda\|_{\8}).
\end{split}
\end{equation*}
This and \eqref{lam} yield {\bf(ii)}.

In the sequel, we intend to verify {\bf(iii)}. From \eqref{h0} and  the semigroup property, we have
\begin{equation*}
\begin{split}
  \nn_\eta P_{t}^0f(x)&=\nn_\eta P_{t/2}^0(P_{t/2}^0f)(x) \\
&=\E\bigg(\frac{(
P_{t/2}^0f)(Z^{x}_{t/2})}{t/2}\int_0^{\frac{t}{2}}\<\e^{-\beta r}\eta,\d W(r)\>\bigg), \ \ t>0, f\in\B_b(\mathbb{R}^d).
\end{split}
\end{equation*}
This further gives that
\begin{equation*}
\begin{split}
&\frac{1}{2}(\nn_{\eta'}\nn_\eta
P_{t}^0f)(x)=\E\bigg(\frac{(\nn_{\e^{-\beta t/2}\eta'}
P_{t/2}^0f)(Z^{x}_{t/2})}{t}\int_0^{t/2}\<\e^{-\beta r}\eta,\d W(r)\>\bigg),\ \ t>0, f\in\B_b(\mathbb{R}^d),
\end{split}
\end{equation*}
where we have used $\nabla_{\eta'}Z_{t/2}^x=\e^{-\beta t/2}\eta'$.
Thus, applying  Cauchy-Schwartz's inequality and  It\^o's isometry and taking \eqref{b7} into consideration, we derive  that
\begin{equation}\label{be}
\begin{split}
|\nn_{\eta'}\nn_\eta P_{t}^0f|^2(x)
&\le\frac{4}{t}|\eta|^2\ff{|\eta'|^2P_{t}^{0}f^{2}(x)}{t}\\
&=\ff{4|\eta|^2|\eta'|^2P_{t}^{0}f^{2}(x)}{t^2},\ \ t>0, f\in\B_b(\mathbb{R}^d).
\end{split}
\end{equation}
\smallskip

Set $\tt h(\cdot):=h(\cdot)-h(\e^{-\beta t}x)$ for fixed $x\in\R^d$ and $h\in\B_b(\R^d)$ which verifies
\begin{equation}\label{a1}
|h(y')-h(y)|\le\tilde{\kappa}|y'-y|^{\tilde{\alpha}},~~~~y',y\in\R^d
\end{equation}
for some $\tilde{\kappa}>0$ and $\tilde{\alpha}\in(0,1)$. Then \eqref{be} implies that
\begin{equation}\label{v1}
\begin{split}
|\nn_{\eta'}\nn_\eta P_{t}^0h|^2(x)=|\nn_{\eta'}\nn_\eta P_{t}^0\tt h|^2(x)& \le
\frac{4|\eta|^{2}|\eta'|^{2}}{t^2}\E|h(Z_t^{x})-h(\e^{-\beta t}x)|^2\\
&\le\frac{4|\eta|^{2}|\eta'|^{2}}{t^2}\tilde{\kappa}^2t^{\tilde{\alpha}}, \ \ t>0.
\end{split}
\end{equation}
where in the second display  we have used that
\begin{equation*}
Z^{x}_t-\e^{-\beta t}x=\int_0^t\e^{-\beta(t-r)}\d W(r),
\end{equation*}
and utilized Jensen's inequality as well as It\^o's isometry. Thus, if we can prove that
\begin{align}\label{ho}
|(b+\nabla_b \mathbf{u}^\lambda)(x)-(b+\nabla_b \mathbf{u}^\lambda)(y)|\leq \tilde{\kappa}|x-y|^{\tilde{\alpha}},~~~~x,y\in\R^d
\end{align}
for some $\tilde{\kappa}>0$ and $\tilde{\alpha}\in(0,1)$, we get from \eqref{b2} and \eqref{v1} that
\begin{equation} \label{2gra}\begin{split}\|\nn^2\mathbf{u}^\lambda\|_{\8}&\le2\tilde{\kappa}\int_0^\infty\frac{\e^{-\lambda t}}{t}t^{\frac{\tilde{\alpha}}{2}}\d t\\
&=2\tilde{\kappa}\lambda^{-\frac{\tilde{\alpha}}{2}}\int_0^\infty\frac{\e^{-t}}{t^{1-\frac{\tilde{\alpha}}{2}}}\d t\\
&=2\tilde{\kappa}\Gamma(\frac{\tilde{\alpha}}{2})\lambda^{-\frac{\tilde{\alpha}}{2}}.
\end{split}\end{equation}

In the remaining, we intend to prove \eqref{ho}. Combining \eqref{b7} and \eqref{be}, we arrive at
\begin{equation}\label{n-n}\begin{split}
&\|\nabla P_t^0f(x)-\nabla P_t^0f(y)\|\\
&\leq \left(\ff{2|x-y|}{t}\wedge\ff{2}{\sqrt{t}}\right)\|f\|_{\infty}, \ \ f\in\B_b(\mathbb{R}^d), t>0, x,y\in\mathbb{R}^d.
\end{split}\end{equation}
Thus, for any $f\in\B_b(\mathbb{R}^d)$, $\lambda>0, x,y\in\mathbb{R}^d,$ it holds that
\begin{equation}\label{n-n'}\begin{split}
&\left\|\int_0^\infty\e^{-\lambda t}\left(\nabla P_t^0f(x)-\nabla P_t^0f(y)\right)\d t\right\|\\
&\leq 4\|f\|_{\infty}\int_0^\infty\e^{-\lambda t}\left(\ff{|x-y|}{t}\wedge\ff{1}{\sqrt{t}}\right)\d t\\
&\leq 4\|f\|_{\infty}\Bigg(\int_0^{|x-y|^2\wedge \e^{-1}}\ff{1}{\sqrt{t}}\d t+|x-y|\int_{|x-y|^2\wedge \e^{-1}}^1\ff{1}{t}\d t+|x-y|\int_{1}^\infty\e^{-\lambda t}\d t\Bigg)\\
&\leq 4\|f\|_{\infty}\left((2+\lambda^{-1})|x-y|+|x-y|\log(|x-y|^{-2}\vee \e)\right)\\
&\leq 4\|f\|_{\infty}(3+\lambda^{-1})|x-y|\log(|x-y|^{-2}+ \e).
\end{split}\end{equation}
 For any $ \lambda>\lambda_0(\delta),$ note from \eqref{b-in}, {\bf(ii)}, \eqref{b2}, \eqref{b7}, \eqref{be} and \eqref{n-n'} that
\begin{align*}
&|(b+\nabla_b \mathbf{u}^\lambda)(x)-(b+\nabla_b \mathbf{u}^\lambda)(y)|\\
&\le(1+\|\nn \mathbf{u}^\lambda\|_{\8})\kappa|x-y|^\alpha+\|b\|_{\8}\|\nn
\mathbf{u}^\ll(x)-\nn \mathbf{u}^\lambda(y)\|{\bf \bf{1}}_{\{|x-y|\ge1\}}\\
&\quad+\|b\|_{\8}\|\nn \mathbf{u}^\ll(x)-\nn \mathbf{u}^\lambda(y)\|{\bf{1}}_{\{|x-y|\le1\}}\\
&\le(1+\delta)\kappa|x-y|^\alpha+2\delta\|b\|_{\8} |x-y|^\alpha{\bf{1}}_{\{|x-y|\ge1\}}\\
&\quad+4(1+\delta)(3+\lambda^{-1})\|b\|_{\8}^2|x-y|^\alpha\times|x-y|^{1-\alpha}\log\Big(\e+\ff{1}{|x-y|^2}\Big){\bf{1}}_{\{|x-y|\le1\}}\\
&\le\left((1+\delta)\kappa+2\delta\|b\|_{\infty}+4(1+\delta)(3+\lambda^{-1})\|b\|^2_{\8}\right)|x-y|^\alpha,\\
\end{align*}
where in the third inequality we have used the fact that the function $[0,1]\ni x\mapsto x^{1-\alpha}\log(\e+\ff{1}{x^2})$ is non-decreasing.
Thus, \eqref{ho} holds for $\tilde{\kappa}=(1+\delta)\kappa+2\delta\|b\|_{\infty}+4(1+\delta)(3+\lambda^{-1})\|b\|^2_{\8}$ and $\tilde{\alpha}=\alpha$. From \eqref{2gra}, we finish the proof.
\end{proof}
\begin{rem} \label{gra} In the multiplicative noise case, precise estimate for $\|\nabla^2\mathbf{u}^\lambda\|_\infty$ is sophisticated, since $\nabla_\eta Z^x_t$ is a random variable. Thus, we only consider the the case where diffusion coefficient is identity operator in this paper.
\end{rem}
\section{Proof of Theorem \ref{TE}}
\begin{lem}\label{supp} Assume {\bf(H1)}-{\bf(H2)} and
$\lambda_B<\beta \e^{-2\beta r_0},$ then
\beq\label{E1''}\beg{split}
\sup_{t\geq 0}\E \|X^\xi_t\|^2_{\infty}<\infty\end{split}
\end{equation}
\end{lem}
\begin{proof}
For simplicity, we denote $X^\xi(t)$ by $X(t)$. It\^{o}'s formula implies that
\begin{align*}
\d|X(t)|^2
=&2\left<X(t),-\beta X(t))\right\>\d t+2\left<X(t),b(X(t))\right\>\d t+2\left\<X(t),\,\d W(t)\right\>\\
&+2\left\<X(t),B(X_t)\right\>\d t+d\,\d t.
\end{align*}
Let $\xi_0(s)=0, s\in[-r_0,0]$. It follows from \eqref{B-Lip} that
\begin{align*}
\left\<X(t),B(X_t)\right\>&\leq |X(t)||B(X_t)|\leq \lambda_B\|X_t\|^2_{\infty}+|B(\xi_0)||X(t)|
\end{align*}
Since $\lambda_B<\beta \e^{-2\beta r_0}$, we can take small enough  $\varepsilon\in(0,1)$ and  $\epsilon\in(0,2\beta )$ such that
\begin{align}\label{coo}\frac{2\lambda_B}{1-\varepsilon}<(2\beta -\epsilon)\e^{-(2\beta -\epsilon)r_0}.
\end{align}

This together with Young's inequality and {\bf(H2)} yields that
\beq\label{NNPp}\begin{split}\d |X(t)|^2
&\le-2\beta |X(t)|^2\d t+\{2(\|b\|_{\infty}+|B(\xi_0)|)|X(t)|+d\}\d t\\
&+2\lambda_B\|X_t\|^2_{\infty}\d t+2\left\<X(t),\d W(t)\right\>\\
&\le(-2\beta +\epsilon)|X(t)|^2\d t+\left\{\frac{1}{\epsilon}(\|b\|_{\infty}+|B(\xi_0)|)^2+d\right\}\d t\\
&+2\lambda_B\|X_t\|^2_{\infty}\d t+2\left\<X(t),\,\d W(t)\right\>.
\end{split}\end{equation}
Thus it is not difficult to see that
\beq\label{NNPp'}\begin{split}\d \e^{(2\beta -\epsilon)t}|X(t)|^2\d t
&\le2\lambda_B\e^{(2\beta -\epsilon)t}\|X_t\|_{\infty}^2\d t\\
&+\e^{(2\beta -\epsilon)t}\left\{\frac{1}{\epsilon}(\|b\|_{\infty}+|B(\xi_0)|)^2+d\right\}\d t\\
&+2\e^{(2\beta -\epsilon)t}\left\<X(t),\,\d W(t)\right\>.
\end{split}\end{equation}
Let $\eta_r=\sup_{s\in[-r_0,r]}\e^{(2\beta -\epsilon)s^+}|X(s)|^2$, then
\beq\label{NNPp'''}\begin{split}
\E\eta_r
&\le\|\xi\|_{\infty}^2+2\lambda_B\e^{(2\beta -\epsilon)r_0}\E\int_0^r\eta_t\d t\\
&+\int_{0}^{r}\e^{(2\beta -\epsilon)t}\left\{\frac{1}{\epsilon}(\|b\|_{\infty}+|B(\xi_0)|)^2+d\right\}\d t\\
&+\E\sup_{s\in[0,r]}\int_0^s2\e^{(2\beta -\epsilon)t}\left\<X(t),\,\d W(t)\right\>
\end{split}\end{equation}
On the other hand, BDG inequality and Young's inequality imply that
\beq\label{NNP0p}\begin{split}&\E\sup_{s\in[0,r]}\int_0^s2\e^{(2\beta -\epsilon)t}\left\<X(t),\,\d W(t)\right\>\\
&\leq \E\left\{\int_0^r16\e^{(4\beta -2\epsilon)t}\left|X(t)\right|^2\d t\right\}^{\frac{1}{2}}\\
&\leq \varepsilon\E\eta_r+\int_0^r\e^{(2\beta -\epsilon)t}\times\frac{4}{\varepsilon}\d t.
\end{split}\end{equation}
Let $\gamma:=\frac{\frac{1}{\epsilon}(\|b\|_{\infty}+|B(\xi_0)|)^2+d+\frac{4}{\varepsilon}}{1-\varepsilon}$. Combining \eqref{NNPp'''} and \eqref{NNP0p}, we have
\beq\label{NNP0111'''}\begin{split}
\E\eta_r
&\le\frac{\|\xi\|_{\infty}^2}{1-\varepsilon}+\frac{2\lambda_B\e^{(2\beta -\epsilon)r_0}}{1-\varepsilon}\E\int_0^r\eta_t\d t+\gamma\int_{0}^{r}\e^{(2\beta -\epsilon)t}\d t.
\end{split}\end{equation}
By \eqref{sup0}, Gronwall's inequality implies that
\beq\label{NNP01'''}\begin{split}
\E\eta_t
&\le\exp\left\{\frac{2\lambda_B\e^{(2\beta -\epsilon)r_0}}{1-\varepsilon}t\right\} \frac{\|\xi\|_{\infty}^2}{1-\varepsilon}\\
&\quad+\gamma \int_{0}^{t}\exp\left\{\frac{2\lambda_B\e^{(2\beta -\epsilon)r_0}}{1-\varepsilon}(t-s)+(2\beta -\epsilon)s\right\}\d s.
\end{split}\end{equation}
Noting that $\E\eta_t\geq \e^{(t-r_0)(2\beta -\epsilon)}\E\|X_t\|_{\infty}^2$, we obtain that
\beq\label{NNP00'''}\begin{split}
\E\|X_t\|_{\infty}^2
&\le\e^{r_0(2\beta -\epsilon)}\exp\left\{\left(\frac{2\lambda_B\e^{(2\beta -\epsilon)r_0}}{1-\varepsilon}-(2\beta -\epsilon)\right)t\right\} \frac{\|\xi\|_{\infty}^2}{1-\varepsilon}\\
&+\e^{r_0(2\beta -\epsilon)}\gamma \int_{0}^{t}\exp\left\{\left(\frac{2\lambda_B\e^{(2\beta -\epsilon)r_0}}{1-\varepsilon}-(2\beta -\epsilon)\right)(t-s)\right\}\d s.
\end{split}\end{equation}
This together with \eqref{coo} gives \eqref{E1''}.
\end{proof}
\begin{proof}[Proof of Theorem \ref{TE}]
(1) Let $X$ and $\bar{X}$ be solutions to \eqref{E1} with $X_0=\xi$, $\bar{X}_0=\eta$, then
\beq\label{E1'''}\beg{split}
&\d X(t)= \{-\beta X(t)+b(X(t))+B( X_t)\}\d t+\d W(t), \ \ X_0=\xi,\\
&\d \bar{X}(t)= \{-\beta \bar{X}(t)+b(\bar{X}(t))+B(\bar{X}_t)\}\d t+\d W(t), \ \ \bar{X}_{0}=\eta.
\end{split}
\end{equation}

For any $\delta\in(0,1)$ and $\lambda> \lambda_0(\delta)$, let $\theta^{\lambda}(x)=x+\mathbf{u}^{\lambda}(x)$. Combining \eqref{b2} and Lemma \ref{L2.1}, we have
\beq\label{PDE}
\frac{1}{2}\mathrm{Tr} (\nabla^2\mathbf{u}^\lambda)+\nabla_{b}\mathbf{u}^\lambda+b+\nabla_{-\beta\cdot}\mathbf{u}^\lambda=\lambda \mathbf{u}^\lambda.
\end{equation}

By \eqref{E1'''}, \eqref{PDE} and It\^o's formula, we have
\begin{align*}
&\d \theta^\lambda(X(t))= \{-\beta X(t)+\lambda \mathbf{u}^\lambda(X(t))+\nabla\theta^\lambda(X(t))B(X_t)\}\d t+ \nabla\theta^\lambda(X(t))\d W(t),  \\
&\d \theta^\lambda(\bar{X}(t))= \{-\beta\bar{X}(t)+\lambda \mathbf{u}^\lambda(\bar{X}(t))+\nabla\theta^\lambda(\bar{X}(t))B(\bar{X}_t)\}\d t+ \nabla\theta^\lambda(\bar{X}(t))\d W(t).
\end{align*}


So, letting $\xi(t)=\theta^\lambda(X(t))-\theta^\lambda(\bar{X}(t))$, we arrive at
\beq\label{NN10}\beg{split}
\d|\xi(t)|^{2}
 =&2\langle -\beta X(t)-(-\beta)\bar{X}(t),\theta^\lambda(X(t))-\theta^\lambda(\bar{X}(t))\rangle\\
 \, &2\lambda\left<\xi(t),\mathbf{u}^{\lambda}(X(t))-\mathbf{u}^{\lambda}(\bar{X}(t))\right\>\d t\\
 &+2\left\<\xi(t),[\nabla\theta^{\lambda}(X(t))B(X_t)-\nabla\theta^{\lambda}(\bar{X}(t))B(\bar{X}_t)]\right\>\d t\\
&+2\left\<\xi(t),[\nabla\theta^{\lambda}(X(t))-\nabla\theta^{\lambda}(\bar{X}(t))]\d W(t)\right\>\\
&+\left\|[\nabla\theta^{\lambda}(X(t))-\nabla\theta^{\lambda}(\bar{X}(t))]\right\|^2_{HS}\,\d t
\end{split}\end{equation}
Firstly, it is easy to see that
\begin{align*}&2\langle -\beta X(t)-(-\beta)\bar{X}(t),\theta^\lambda(X(t))-\theta^\lambda(\bar{X}(t))\rangle\\
=&2\langle -\beta X(t)-(-\beta)\bar{X}(t),X(t)-\bar{X}(t)+\mathbf{u}^\lambda(X(t))-\mathbf{u}^\lambda(\bar{X}(t))\rangle\\
\leq&-2\beta|X(t)-\bar{X}(t)|^2+2\beta\delta |X(t)-\bar{X}(t)|^2
\end{align*}
By Lemma \ref{L2.1} and {\bf(H2)}, it holds that
\begin{align*}
&2\left\<\xi(t),[\nabla\theta^{\lambda}(X(t))B(X_t)-\nabla\theta^{\lambda}(\bar{X}(t))B(\bar{X}_t)]\right\>\\
&\leq 2(1+\delta)|X(t)-\bar{X}(t)|\|\nabla\theta^{\lambda}(X(t))-\nabla\theta^{\lambda}(\bar{X}(t))\||B(X_t)|\\
&+2(1+\delta)|X(t)-\bar{X}(t)|\|\nabla\theta^{\lambda}(\bar{X}(t))\||B(X_t)-B(\bar{X}_t)|\\
&\leq 2(1+\delta)\Upsilon(\lambda,\delta)\|B\|_{\infty}|X(t)-\bar{X}(t)|^2+2(1+\delta)^2\lambda_B\|X_t-\bar{X}_t\|_\infty^2,
\end{align*}
and
\beq\label{XPP1}2\lambda |\xi(t)|\cdot|\mathbf{u}^{\lambda}(X(t))-\mathbf{u}^{\lambda}(\bar{X}(t))|\le 2\lambda(1+\delta)\delta\ |X(t)-\bar{X}(t)|^2.\end{equation}
Moreover,
\beq\label{XPP2}\beg{split} \left\|[\nabla\theta^{\lambda}(X(t))-\nabla\theta^{\lambda}(\bar{X}(t))]\right\|^2_{HS} \le d\Upsilon(\lambda,\delta)^2|X(t)-\bar{X}(t)|^2.\end{split}\end{equation}
Let
$$\Sigma(\lambda):=\frac{1}{(1-\delta)^2}\left(2(1+\delta)\delta\lambda +2\beta\delta+2(1+\delta)^2\lambda_B+2(1+\delta)\Upsilon(\lambda,\delta) \|B\|_{\infty}+d\Upsilon(\lambda,\delta)^2\right).$$
Since $\|\nabla \theta^\lambda(x)\|\geq (1-\delta)$ for any $x\in\mathbb{R}^d$, we have
\beq\label{NNP}\begin{split}\d |X(t)-\bar{X}(t)|^2
&\le\frac{-2\beta}{(1-\delta)^2}|X(t)-\bar{X}(t)|^2\d t\\
&+\Sigma(\lambda)\|X_t-\bar{X}_t\|_{\infty}^2\d t\\
&+\frac{2}{(1-\delta)^2}\left\<\xi(t),[\nabla\theta^{\lambda}(X(t))-\nabla\theta^{\lambda}(\bar{X}(t))]\d W(t)\right\>.
\end{split}\end{equation}
Similarly to \eqref{NNPp'}, it holds that
\begin{align*}\d \e^{\frac{2\beta}{(1-\delta)^2} t}|X(t)-\bar{X}(t)|^2\d t
&\le\Sigma(\lambda)\e^{\frac{2\beta}{(1-\delta)^2} t}\|X_t-\bar{X}_t\|_{\infty}^2\\
&+\frac{2}{(1-\delta)^2}\e^{\frac{2\beta}{(1-\delta)^2} t}\left\<\xi(t),[\nabla\theta^{\lambda}(X(t))-\nabla\theta^{\lambda}(\bar{X}(t))]\d W(t)\right\>.
\end{align*}

Set $\gamma_r=\sup_{s\in[-r_0,r]}\e^{\frac{2\beta}{(1-\delta)^2} s^{+}}|X(s)-\bar{X}(s)|^2$ and we get
\beq\label{NNP'''}\begin{split}
\E\gamma_r
&\le\|\xi-\eta\|_{\infty}^2+\Sigma(\lambda)\e^{\frac{2\beta}{(1-\delta)^2} r_0}\E\int_0^r\gamma_t\d t\\
&+\E\sup_{s\in[0,r]}\int_0^s\frac{2}{(1-\delta)^2}\e^{\frac{2\beta}{(1-\delta)^2} t}\left\<\xi(t),[\nabla\theta^{\lambda}(X(t))-\nabla\theta^{\lambda}(\bar{X}(t))]\d W(t)\right\>.
\end{split}\end{equation}
Again, BDG inequality and Young's inequality yield that for any $\varepsilon\in(0,1)$,
\beq\label{NNP0}\begin{split}&\E\sup_{s\in[0,r]}\int_0^s\frac{2}{(1-\delta)^2}\e^{\frac{2\beta}{(1-\delta)^2} t}\left\<\xi(t),[\nabla\theta^{\lambda}(X(t))-\nabla\theta^{\lambda}(\bar{X}(t))]\d W(t)\right\>\\
&\leq 2\E\left\{\int_0^r\frac{4}{(1-\delta)^4}\e^{\frac{4\beta}{(1-\delta)^2}t}\left|[\nabla\theta^{\lambda}(X(t))-\nabla\theta^{\lambda}(\bar{X}(t))]^\ast\xi(t)\right|^2\d t\right\}^{\frac{1}{2}}\\
&\leq 2\E\sqrt{ \frac{4}{(1-\delta)^4}\Upsilon(\lambda,\delta)^2\int_0^r\e^{\frac{4\beta}{(1-\delta)^2}t}|X(t)-\bar{X}(t)|^4\d t}\\
&\leq \varepsilon\E\gamma_r+\frac{4}{(1-\delta)^4\varepsilon}\Upsilon(\lambda,\delta)^2\E\int_0^r\gamma_t\d t.
\end{split}\end{equation}
Substituting \eqref{NNP0} into \eqref{NNP'''}, we obtain
\beq\label{NNP0'''}\begin{split}
\E\gamma_r
&\le\frac{\|\xi-\eta\|_{\infty}^2}{1-\varepsilon}+\frac{\Sigma(\lambda) \e^{\frac{2\beta}{(1-\delta)^2} r_0}+\frac{4}{(1-\delta)^4\varepsilon}\Upsilon(\lambda,\delta)^2}{1-\varepsilon}\E\int_0^r\gamma_t\d t.
\end{split}\end{equation}
Again by \eqref{sup0}, Gronwall's inequality implies that
\beq\label{NNP11'''}\begin{split}
\E\gamma_t
&\le\exp\left\{\frac{\Sigma(\lambda)\e^{\frac{2\beta}{(1-\delta)^2} r_0}+\frac{4}{(1-\delta)^4\varepsilon}\Upsilon(\lambda,\delta)^2}{1-\varepsilon}t\right\}\frac{\|\xi-\eta\|_{\infty}^2}{1-\varepsilon}.
\end{split}\end{equation}
Noting that $\E\gamma_t\geq \e^{(t-r_0)\frac{2\beta}{(1-\delta)^2} }\E\|X_t-\bar{X}_t\|_{\infty}^2$, we obtain
\begin{equation*}\begin{split}
&\E\|X_t-\bar{X}_t\|_{\infty}^2\\
&\le\e^{\frac{2\beta r_0}{(1-\delta)^2} }\exp\left\{\e^{\frac{2\beta r_0}{(1-\delta)^2} }\left(\frac{\Sigma(\lambda)+\frac{4\Upsilon(\lambda,\delta)^2}{(1-\delta)^4\varepsilon}}{1-\varepsilon}-\frac{2\beta}{(1-\delta)^2}  \e^{-\frac{2\beta r_0}{(1-\delta)^2} }\right)t\right\}\frac{\|\xi-\eta\|_{\infty}^2}{1-\varepsilon}.
\end{split}\end{equation*}
Thus, combining the definition of $\Sigma(\lambda)$ and $\Lambda(\lambda,\varepsilon,\delta)$, we prove (1).

(2) Now, if there exists $\tilde{\delta}, \tilde{\varepsilon}\in(0,1), \tilde{\lambda}> \lambda_0(\tilde{\delta})$ such that
$$\Lambda(\tilde{\lambda},\tilde{\varepsilon},\tilde{\delta})<2\beta \e^{-\frac{2\beta}{(1-\tilde{\delta})^2} r_0},$$
then there exists constants $\kappa_0,\kappa_2>0$ such that
\beq\label{NNP000'''}\begin{split}
\mathbb{W}_2(P_t(\xi,\cdot,P_t(\eta,\cdot))^2\leq \E\|X_t-\bar{X}_t\|_{\infty}^2
&\le\kappa_0\e^{-\kappa_2t}\|\xi-\eta\|_{\infty}^2.
\end{split}\end{equation}
This yields that for any $0<t<s$,
\beq\label{cauchy}\begin{split}
\E\|X^\xi_t-X^\xi_s\|_{\infty}^2=\E\left\|X^\xi_t-X^{X^\xi_{s-t}}_t\right\|_{\infty}^2
&\le\kappa_0\e^{-\kappa_2t}\mathbb{E}\|\xi-X^\xi_{s-t}\|_{\infty}^2,
\end{split}\end{equation}
Noting that $\Lambda(\tilde{\lambda},\tilde{\varepsilon},\tilde{\delta}) \geq \frac{2\lambda_B}{1-\tilde{\varepsilon}}$ and $$\Lambda(\tilde{\lambda},\tilde{\varepsilon},\tilde{\delta})<2\beta \e^{-\frac{2\beta}{(1-\tilde{\delta})^2} r_0},$$ we conclude that $$\lambda_B<\beta \e^{-2\beta r_0}.$$
According to Lemma \ref{supp}, it holds that $$\sup_{r\geq 0}\mathbb{E}\|X^\xi_{r}\|_{\infty}^2<\infty.$$
Thus, \beq\label{cauchy0}\begin{split}
\mathbb{W}_2(P_t(\xi,\cdot),P_s(\xi,\cdot))^2\le\E\|X^\xi_t-X^\xi_s\|_{\infty}^2\leq \kappa_1(\xi)\e^{-\kappa_2t}
\end{split}\end{equation}
holds for some positive constant $\kappa_1(\xi)$ depending on $\xi$.
So, there exists a probability measure $\mu_\xi$ such that
\beq\label{con}\begin{split}
\mathbb{W}_2(P_t(\xi,\cdot),\mu_\xi)^2\leq \kappa_1(\xi)\e^{-\kappa_2t}.
\end{split}\end{equation}
It remains to prove that $\mu_\xi$ does not depend on $\xi$. For any $\xi,\eta\in\C$,
\beq\label{cauchy1}\begin{split}
\mathbb{W}_2(\mu_\xi,\mu_\eta)^2&\leq \mathbb{W}_2(P_t(\xi,\cdot),\mu_\xi)^2+\mathbb{W}_2(P_t(\xi,\cdot),P_t(\eta,\cdot))^2+\mathbb{W}_2(P_t(\eta,\cdot),\mu_\eta)^2\\
&\leq \kappa_1(\xi)\e^{-\kappa_2t}+\kappa_0\e^{-\kappa_2t}\|\xi-\eta\|_{\infty}^2+\kappa_1(\eta)\e^{-\kappa_2t}.
\end{split}\end{equation}
Letting $t\to\infty$, we obtain $\mu_\xi=\mu_\eta.$

Next, for any $t_0>r_0$ and $t>t_0$, by the semigroup property, \eqref{Ent00} and \eqref{cauchy0}, we have
\begin{align*}\Ent(P_{t}(\xi,\cdot)|P_{t}(\eta,\cdot))&=\Ent(P_{t-t_0}(\xi,\cdot)P_{t_0}|P_{t-t_0}(\eta,\cdot)P_{t_0})\\
&\leq  \frac{C(t_0)}{(t_0-r_0)\wedge 1}\mathbb{W}_2(P_{t-t_0}(\xi,\cdot),P_{t-t_0}(\eta,\cdot) )^2\\
&\leq \frac{C(t_0)}{(t_0-r_0)\wedge 1}\kappa_0\e^{\kappa_2t_0}\e^{-\kappa_2t}\|\xi-\eta\|_{\infty}^2. \end{align*}
Thus \eqref{ETX} implies that for any $t>t_0$,
\beq\label{TT'}\|P_t(\xi,\cdot)-P_t(\eta,\cdot)\|_{var}^2\le \frac{C(t_0)}{2(t_0-r_0)\wedge 2}\kappa_0\e^{\kappa_2t_0}\e^{-\kappa_2t}\|\xi-\eta\|_{\infty}^2. \end{equation}
Combining \eqref{Ent00}, \eqref{TT00}, \eqref{con} and the semigroup property, since $\mu$ is the invariant probability measure, we have
\begin{align*}\Ent(P_{t}(\xi,\cdot)|\mu)&=\Ent(P_{t-t_0}(\xi,\cdot)P_{t_0}|\mu P_{t_0})\\
&\leq  \frac{C(t_0)}{(t_0-r_0)\wedge 1}\mathbb{W}_2(P_{t-t_0}(\xi,\cdot),\mu )^2\\
&\leq \frac{C(t_0)}{(t_0-r_0)\wedge 1}\kappa_1(\xi)\e^{\kappa_2t_0}\e^{-\kappa_2t}. \end{align*}
and
\begin{align*}\|P_t(\xi,\cdot)-\mu\|_{var}^2\le \frac{C(t_0)}{2(t_0-r_0)\wedge 2}\kappa_1(\xi)\e^{\kappa_2t_0}\e^{-\kappa_2t}. \end{align*}
Thus, we complete the proof.
\end{proof}

\paragraph{Acknowledgement.} The author would like to thank Professor Feng-Yu Wang and Shao-Qin Zhang for helpful comments.

\end{document}